\documentclass[10.5pt,a4paper]{amsart}
\usepackage{etex} 
\usepackage[latin1]{inputenc}
\usepackage[T1]{fontenc}
\usepackage[right=3cm,left=3cm,top=2cm,bottom=2cm]{geometry}
\usepackage{ifthen}
\usepackage{enumerate}
\usepackage{amsmath}
\usepackage{amsthm}
\usepackage{amsfonts}
\usepackage{amssymb}
\usepackage{latexsym}
\usepackage{ae}
\usepackage{xcolor} 
\usepackage{graphics,amsmath,amssymb}
\usepackage[np]{numprint}
\usepackage{fourier} 
\usepackage{hyperref}

\npdecimalsign{\ensuremath{.}}

\title{On a conjecture of B. C. Kellner}
\author{Olivier Bordell\`{e}s}
\address{2 all\'{e}e de la combe \\ 43000 Aiguilhe \\ France}
\email{borde43@wanadoo.fr}
\date{}
\dedicatory{}

\theoremstyle{plain}
\newtheorem{theorem}{Theorem}
\newtheorem{corollary}[theorem]{Corollary}
\newtheorem{lemma}[theorem]{Lemma}

\theoremstyle{definition}

\theoremstyle{remark}

\newcommand{\Z}{\mathbb {Z}}

\newcommand{\R}{\mathbb {R}}

\DeclareMathOperator{\li}{Li}

\makeatletter
\@namedef{subjclassname@2010}{%
  \textup{2010} Mathematics Subject Classification}
\makeatother

\begin{document}

\begin{abstract}
The aim of this note is a proof of a recent conjecture of Kellner concerning the number of distinct prime factors of a particular product of primes. The proof uses profound results from analytic number theory, such as Granville-Ramar\'{e}'s estimate of an exponential sum over primes.
\end{abstract}


\maketitle

\thispagestyle{myheadings}
\font\rms=cmr8 
\font\its=cmti8 
\font\bfs=cmbx8

\section{Introduction and main result}

Let $n \in \Z_{\geqslant 1}$. In a recent paper \cite{kell}, Kellner studies the product
$$\mathfrak{p}_n := \prod_{\substack{p \; \textrm{prime} \\ s_p(n) \geqslant p}} p$$
where $s_p(n)$ stands for the sum of the digits in the base-$p$ expansion of $n$. Recall that, from Legendre's formula, $s_p(n) = n - (p-1)v_p(n!)$. As it is shown in \cite{kell,kells}, the values of this product are closely related to the denominators of the Bernoulli polynomials. In \cite{kell}, the author proves that 
$$\omega \left( \mathfrak{p}_n^+ \right) < \sqrt{n} \quad \textrm{and} \quad \omega \left( \mathfrak{p}_n^- \right) \leqslant \pi \left( \sqrt{n} \right)$$
where
$$\mathfrak{p}_n = \prod_{\substack{p < \sqrt{n} \\ s_p(n) \geqslant p}} p \times \prod_{\substack{p > \sqrt{n} \\ s_p(n) \geqslant p}} p := \mathfrak{p}_n^- \times \mathfrak{p}_n^+$$
and points out that the trivial bound $\omega \left( \mathfrak{p}_n^- \right) \leqslant \pi \left( \sqrt{n} \right)$ is sharp. On the other hand, the first bound concerning the number of distinct prime factors of $\mathfrak{p}_n^+$, coming from the identity (see \cite[Theorem~4]{kell})
\begin{equation}
   \omega \left( \mathfrak{p}_n^+ \right) = \sum_{\substack{p > \sqrt{n} \\ \left \lfloor \frac{n-1}{p-1} \right \rfloor > \left \lfloor \frac{n}{p} \right \rfloor}} 1 \label{e1}
\end{equation}
does not seem to be the best one and, on the basis of extended computations, the author surmises that there exists $\kappa > 1$ such that, for $n \to \infty$
\begin{equation}
   \omega \left( \mathfrak{p}_n^+ \right) \sim \kappa \, \frac{\sqrt{n}}{\log n}. \tag{C} \label{c} 
\end{equation}

In this note, we provide a proof of Conjecture~\eqref{c} which is a consequence of the more precise following result.

\begin{theorem}
\label{t1}
For any integer $n \in \Z_{\geqslant 1}$ sufficiently large
$$\omega \left( \mathfrak{p}_n^+ \right) = n \mathrm{E}_1 \left( \log \sqrt{n} \right)  + O \left( n^{1/2} \delta \left( \sqrt{n} \right) \right)$$
where the function $\delta$ is given in \eqref{e2} below and
\begin{equation}
   \textrm{E}_1 (x) = \int_x^\infty \frac{e^{-t}}{t} \, \mathrm{d}t \quad \left( x > 0 \right) \label{e1}
\end{equation}
is the exponential integral.
\end{theorem}

Since, for any $N \in \Z_{\geqslant 1}$
$$\textrm{E}_1(x) = \frac{e^{-x}}{x} \sum_{m=0}^{N-1} \frac{(-1)^m m!}{x^m} + (-1)^N N! \int_x^\infty \frac{e^{-t}}{t^{N+1}} \, \textrm{d}t$$
we deduce immediately the following estimate.

\begin{corollary}
For any integer $n \in \Z_{\geqslant 1}$ sufficiently large and any $N \in \Z_{\geqslant 1}$
$$\omega \left( \mathfrak{p}_n^+ \right) = \frac{2 \sqrt{n}}{\log n} - \frac{4 \sqrt{n}}{(\log n)^2} + \frac{16 \sqrt{n}}{(\log n)^3} - \frac{96 \sqrt{n}}{(\log n)^4} + \dotsb + (-1)^{N-1} \frac{2^N (N-1)! \sqrt{n}}{(\log n)^N} + O \left( \frac{2^{N+1} N! \sqrt{n}}{(\log n)^{N+1}} \right).$$
\end{corollary}

\section{Notation.} $n \in \Z_{\geqslant 1}$ is a large integer, say $n \geqslant 10^{20}$, and $p$ always denotes a prime number. For any $x \in \R$, $e(x) := e^{2 i \pi x}$ and $\psi(x) := x - \lfloor x \rfloor - \frac{1}{2}$. 
For $m \in \Z_{\geqslant 1}$ and real numbers $1 \leqslant a < b \leqslant m$
$$\omega(m;a,b) := \sum_{\substack{p \mid m \\ a < p \leqslant b}} 1.$$
Finally
\begin{equation}
   \delta(x) := e^{- \np{0.2098} (\log x)^{3/5} (\log \log x)^{-1/5}} \quad \left( x > e \right). \label{e2}
\end{equation}

\section{Tools}

\begin{lemma}
\label{le2}
For any $\varepsilon > 0$, $x \geqslant 2$ and $1 \leqslant x^{\varepsilon} \leqslant y \leqslant z < x$
$$\sum_{z < p \leqslant x} \left( \left \lfloor \frac{x+y}{p} \right \rfloor - \left \lfloor \frac{x}{p} \right \rfloor \right) < 2(y+1) \varepsilon^{-1}.$$
\end{lemma}

\begin{proof}
Interchanging the order of summation, we get
$$\sum_{z < p \leqslant x} \left( \left \lfloor \frac{x+y}{p} \right \rfloor - \left \lfloor \frac{x}{p} \right \rfloor \right) = \sum_{x < m \leqslant x+y} \omega(m;z,x).$$
Writing each $m \in \left( x,x+y \right]$ as $m=ab$ with $(a,b)=1$, $p \mid a \Longrightarrow p \in \left( z,x \right] $ and $p \mid b \Longrightarrow p \not \in \left( z,x \right]$, we infer
$$m \geqslant a > z^{\omega(m;z,x)} \Longleftrightarrow \omega(m;z,x) < \frac{\log m}{\log z} < \frac{\log 2x}{\log \left( x^{\varepsilon} \right) } \leqslant 2 \varepsilon^{-1}$$
and hence
$$\sum_{z < p \leqslant x} \left( \left \lfloor \frac{x+y}{p} \right \rfloor - \left \lfloor \frac{x}{p} \right \rfloor \right) < 2\varepsilon^{-1} \sum_{x < m \leqslant x+y} 1 \leqslant 2(y+1) \varepsilon^{-1}$$
as asserted.
\end{proof}

\begin{corollary}
\label{cor3}
$\alpha \in \left] n^{-7/16},1 \right[$. Then
$$\sum_{\alpha^{-1}\sqrt{n} < p \leqslant n} \left( \left \lfloor \frac{n}{p} + \frac{n-p}{p(p-1)} \right \rfloor - \left \lfloor \frac{n}{p} \right \rfloor \right) \ll \alpha \sqrt{n}.$$
\end{corollary}

\begin{proof}
Since the function $p \longmapsto \frac{n-p}{p-1}$ is non-increasing, if $p > \alpha^{-1} \sqrt{n}$, then
$$\alpha \sqrt{n} - \frac{n-p}{p-1} \geqslant \alpha \sqrt{n} - \frac{n-\alpha^{-1} \sqrt{n}}{\alpha^{-1} \sqrt{n}-1} = \frac{\sqrt{n}(1+\alpha)(1-\alpha)}{\sqrt{n}-\alpha} > 0.$$
Consequently, the sum of the left-hand side does not exceed
$$\leqslant \sum_{\alpha^{-1}\sqrt{n} < p \leqslant n} \left( \left \lfloor \frac{n + \alpha \sqrt{n}}{p} \right \rfloor - \left \lfloor \frac{n}{p} \right \rfloor \right)$$
and the proof is achieved with the use of Lemma~\ref{le2} with $x=n$, $y= \alpha \sqrt{n}$, $z = \alpha^{-1} \sqrt{n}$ and $\varepsilon = \frac{1}{16}$.
\end{proof}

\begin{lemma}
\label{le4}
Let $M \in \Z_{\geqslant 1}$ and $f : \left[ M,2M \right] \longrightarrow \R$ be any map. For any $H \in \Z_{\geqslant 1}$
$$\left | \sum_{ M < p \leqslant 2M} \psi(f(p)) \right | \ll \frac{M}{H} + \sum_{h \leqslant H} \frac{1}{h} \left | \sum_{ M < p \leqslant 2M} e(hf(p)) \right |.$$
\end{lemma}

\begin{proof}
See \cite[Corollary~6.2]{borl}.
\end{proof}

\begin{lemma}
\label{le5}
If $M \leqslant \frac{1}{5} \, x^{3/5}$, then, for any $M_1 \in \left( M,2M \right]$
$$\left | \sum_{M < m \leqslant M_1} \Lambda(m) e \left( \frac{x}{m} \right) \right | < 17 \left( x^2 M^{19} \right)^{1/24} (\log 16M)^{11/4}.$$
\end{lemma}

\begin{proof}
See \cite[Theorem~9]{gra} with $k=2$.
\end{proof}

\begin{lemma}
\label{le6}
For any real number $t > 1$
$$\sum_{p > t} \frac{1}{p(p-1)} = \mathrm{E}_1(\log t) + O \left( t^{-1} \delta(t) \right).$$
\end{lemma}

\begin{proof}
By partial summation and the Prime Number Theorem, for instance in the form given in \cite{for}
\begin{eqnarray*}
   \sum_{p > t} \frac{1}{p(p-1)} &=& \sum_{p > t} \frac{1}{p^2} + \sum_{p > t} \frac{1}{p^2(p-1)} \\
   &=& - \frac{\pi(t)}{t^2} + 2 \int_t^\infty \frac{\pi(u)}{u^3} \, \textrm{d}u + O \left( \frac{1}{t^2 \log t} \right) \\
   &=& - \frac{1}{t^2} \left( \li(t) + O \left( t \delta(t) \right) \right) + 2 \int_t^\infty \frac{1}{u^3} \left( \li(u) + O \left( u \delta(u) \right) \right) \, \textrm{d}u + O \left( \frac{1}{t^2 \log t} \right) \\
   &=& - \frac{\li(t)}{t^2} + 2 \int_t^\infty \frac{\li(u)}{u^3} \, \textrm{d}u + O \left( t^{-1} \delta(t) \right) \\
\end{eqnarray*}
and integrating by parts
$$- \frac{\li(t)}{t^2} + 2 \int_t^\infty \frac{\li(u)}{u^3} \, \textrm{d}u = \int_t^\infty \frac{\textrm{d}u}{u^2 \log u} = \int_{\log t}^\infty \frac{e^{-v}}{v} \, \textrm{d}v = \mathrm{E}_1(\log t)$$
as asserted.
\end{proof}

\section{Proof of Theorem~\ref{t1}}

\subsection{First step}

Notice that, if $\sqrt{n} < p \leqslant n$, then
$$0 \leqslant \frac{n-p}{p-1} < \sqrt{n}$$
and hence
$$\frac{n-1}{p-1} - \frac{n}{p} = \frac{1}{p} \frac{n-p}{p-1} < \frac{\sqrt{n}}{p} < 1$$
so that we get from \eqref{e1}
$$\omega \left( \mathfrak{p}_n^+ \right) = \sum_{\sqrt{n} < p \leqslant n} \left( \left \lfloor \frac{n-1}{p-1} \right \rfloor - \left \lfloor \frac{n}{p} \right \rfloor \right) = \sum_{\sqrt{n} < p \leqslant n} \left( \left \lfloor \frac{n}{p} + \frac{n-p}{p(p-1)} \right \rfloor - \left \lfloor \frac{n}{p} \right \rfloor \right).$$
Split the sum into two subsums as
$$\omega \left( \mathfrak{p}_n^+ \right) = \left( \sum_{\sqrt{n} < p \leqslant \sqrt{n} \, \left( \delta \left( \sqrt{n} \right) \right)^{-1}} + \sum_{\sqrt{n} \, \left( \delta \left( \sqrt{n} \right) \right)^{-1} < p \leqslant n} \right) \left( \left \lfloor \frac{n}{p} + \frac{n-p}{p(p-1)} \right \rfloor - \left \lfloor \frac{n}{p} \right \rfloor \right) = S_1 + S_2$$

\subsection{The sum $S_2$}

We use Corollary~\ref{cor3} with $\alpha = \delta \left( \sqrt{n} \right)$ giving immediately

\begin{equation}
   S_2 \ll \sqrt{n} \, \delta \left( \sqrt{n} \right). \label{e3}
\end{equation}

\subsection{The sum $S_1$}

First write
$$S_1 = \sum_{\sqrt{n} < p \leqslant \sqrt{n} \, \left( \delta \left( \sqrt{n} \right) \right)^{-1}} \frac{n-p}{p(p-1)} - \sum_{\sqrt{n} < p \leqslant \sqrt{n} \, \left( \delta \left( \sqrt{n} \right) \right)^{-1}} \left( \psi \left( \frac{n}{p} + \frac{1}{p} \frac{n-p}{p-1} \right) - \psi \left( \frac{n}{p} \right) \right) :=S_{11} - S_{12}$$
say.

\subsubsection{The main term}

From Lemma~\ref{le6}
\begin{eqnarray*}
   \sum_{\sqrt{n} < p \leqslant \sqrt{n} \left( \delta (\sqrt{n}) \right)^{-1}} \frac{1}{p(p-1)} &=& \mathrm{E}_1 \left( \log \sqrt{n} \right) - \mathrm{E}_1 \left( \log \left( \sqrt{n} \left( \delta (\sqrt{n}) \right)^{-1}\right)  \right) + O \left( n^{-1/2} \delta \left( \sqrt{n} \right) \right)
\end{eqnarray*}
and the inequalities
$$\frac{e^{-x}}{x+1} < \mathrm{E}_1 (x) < \frac{e^{-x}}{x} \quad \left( x > 0 \right)$$
imply that
$$\mathrm{E}_1 \left( \log \left( \sqrt{n} \left( \delta (\sqrt{n}) \right)^{-1}\right)  \right) \asymp n^{-1/2} \frac{\delta \left( \sqrt{n} \right)}{\log \left(n \left( \delta (\sqrt{n}) \right)^{-2} \right)} \ll n^{-1/2} \delta \left( \sqrt{n} \right)$$
and hence
$$\sum_{\sqrt{n} < p \leqslant \sqrt{n} \left( \delta (\sqrt{n}) \right)^{-1}} \frac{1}{p(p-1)} = \mathrm{E}_1 \left( \log \sqrt{n} \right)  + O \left( n^{-1/2} \delta \left( \sqrt{n} \right) \right).$$
Therefore
\begin{eqnarray}
   S_{11} &=& n \sum_{\sqrt{n} < p \leqslant \sqrt{n} \left( \delta (\sqrt{n}) \right)^{-1}} \frac{1}{p(p-1)} - \sum_{\sqrt{n} < p \leqslant \sqrt{n} \left( \delta (\sqrt{n}) \right)^{-1}} \frac{1}{p-1} \notag \\
   &=& n \sum_{\sqrt{n} < p \leqslant \sqrt{n} \left( \delta (\sqrt{n}) \right)^{-1}} \frac{1}{p(p-1)} + O \left( \log \log n \right) \notag \\
   &=& n \mathrm{E}_1 \left( \log \sqrt{n} \right)  + O \left( n^{1/2} \delta \left( \sqrt{n} \right) \right). \label{e4}
\end{eqnarray}

\subsubsection{The error term}

It remains to prove that, for $n$ sufficiently large
\begin{equation}
   \left | S_{12} \right | \ll n^{1/2} \delta \left( \sqrt{n} \right). \label{e5}
\end{equation}
This estimate follows from the next result.

\begin{lemma}
\label{le7}
For any integer $n \geqslant 3$ sufficiently large
$$\left | S_{12} \right | \ll n^{49/100} (\log n)^{67/25}.$$
\end{lemma}

\begin{proof}
Split the interval $\left( \sqrt{n} , \sqrt{n} \, \left( \delta (\sqrt{n}) \right)^{-1} \right]$ into $O (\log n)$ dyadic subintervals of the shape $\left( M,2M\right]$ so that
$$\left | \sum_{\sqrt{n} < p \leqslant \sqrt{n} \, \left( \delta (\sqrt{n}) \right)^{-1}} \psi \left( \frac{n}{p} + g(p) \right) \right | \ll \underset{\sqrt{n} < M \leqslant \sqrt{n} \, \left( \delta (\sqrt{n}) \right)^{-1}}{\max} \left | \sum_{M < p \leqslant 2M} \psi \left( \frac{n}{p} + g(p) \right) \right | \log n$$
where either $g(p) = 0$ or $g(p) := \frac{n-p}{p(p-1)}$. From Lemma~\ref{le4}
$$\left | \sum_{M < p \leqslant 2M} \psi \left( \frac{n}{p} + g(p) \right) \right | \ll \frac{M}{H} + \sum_{h \leqslant H} \frac{1}{h} \left | \sum_{ M < p \leqslant 2M} e \left( \frac{nh}{p} \right) e(hg(p)) \right |$$
and, for any $p \in \left( M,2M \right]$, we have
$$\left | hg(p) \right | \leqslant \frac{nh}{p(p-1)} \leqslant \frac{2nh}{M^2}$$
so that, by Abel summation
$$\left | \sum_{ M < p \leqslant 2M} e \left( \frac{nh}{p} \right) e(hg(p)) \right | \ll \left( 1 + \frac{nh}{M^2} \right) \max_{M < M_1 \leqslant 2M} \left | \sum_{ M < p \leqslant M_1} e \left( \frac{nh}{p} \right) \right |.$$
Now by Abel summation and Lemma~\ref{le5}
\begin{eqnarray*}
   \left | \sum_{M < p \leqslant M_1} e \left( \frac{nh}{p} \right) \right | & \leqslant & \frac{2}{\log M} \underset{M < M_2 \leqslant M_1}{\max} \left | \sum_{M \leqslant p \leqslant M_2} \left( \log p \right) e \left( \frac{nh}{p} \right) \right | \\
   & \ll & \frac{1}{\log M} \left( \underset{M < M_2 \leqslant M_1}{\max} \left | \sum_{M \leqslant m \leqslant M_2} \Lambda(m) e \left( \frac{nh}{m} \right) \right | + M_1^{1/2} \right) \\
   & \ll & \left( h^2n^2 M^{19} \right)^{1/24} (\log M)^{7/4} + \frac{M_1^{1/2}}{\log M}.
\end{eqnarray*}
Consequently
\begin{eqnarray*}
   \left | \sum_{M < p \leqslant 2M} \psi \left( \frac{n}{p} + g(p) \right) \right | & \ll & \frac{M}{H} + \sum_{h \leqslant H} \frac{1}{h}\left( 1 + \frac{nh}{M^2} \right) \left( \left( h^2n^2 M^{19} \right)^{1/24} (\log M)^{7/4} + \frac{M^{1/2}}{\log M} \right) \\
   & \ll & \frac{M}{H} + \left\lbrace \left( (Hn)^{26} M^{-29} \right)^{1/24} + \left( n^2H^2M^{19} \right)^{1/24} \right\rbrace (\log M)^{7/4} \\
   & & {} + \frac{nH}{M^{3/2} \log M} + \frac{M^{1/2} \log H}{\log M}.
\end{eqnarray*}
Choose
$$H = \left \lfloor \left( M^{53} n^{-26} \right)^{1/50} \left( \tfrac{e}{42} \log M \right)^{-21/25} \right \rfloor$$
so that
\begin{eqnarray*}
   \left | \sum_{M < p \leqslant 2M} \psi \left( \frac{n}{p} + g(p) \right) \right | & \ll & \left( n^{25} M^{-3} \right)^{1/50} (\log M)^{21/25} + \left( n M^{22} \right)^{1/25} (\log M)^{42/25} \\
   & & {} + \left( n^{12} M^{-11} \right)^{1/25}(\log M)^{-46/25} + M^{1/2}
\end{eqnarray*}
and hence
\begin{eqnarray*}
   \left | S_{12} \right | & \ll & \left( n^{49/100}  + n^{48/100} \left( \delta \left( \sqrt{n} \right) \right)^{-22/25} \right) (\log n)^{67/25} + n^{13/50} (\log n)^{-21/25} + n^{1/4} \left( \delta \left( \sqrt{n} \right) \right)^{-1/2} \log n\\
   & \ll & n^{49/100} (\log n)^{67/25}
\end{eqnarray*}
whenever $n \geqslant 3$ is sufficiently large, concluding the proof.
\end{proof}

\subsection{Completion of the proof of Theorem~\ref{t1}}

Follows at once from \eqref{e3}, \eqref{e4} and \eqref{e5}.
\qed


\begin{thebibliography}{9}
   \bibitem{borl} O. Bordell\`{e}s, \textit{Arithmetic Tales}, Springer, \textsc{utx}, 2012.
   \bibitem{for} K. Ford, Vinogradov's integral and bounds for the Riemann zeta function, \textit{Proc. London Math. Soc.} \textbf{85} (2002), 565--633.
   \bibitem{gra} A. Granville \& O. Ramar\'{e}, Explicit bounds on exponential sums and the scarcity of squarefree binomial coefficients, \textit{Mathematika} \textbf{43} (1996), 73--107.
   \bibitem{kell} B. C. Kellner, On a product of certain primes, \textit{J. Number Theory} \textbf{179} (2017), 149--164.
   \bibitem{kells} B. C. Kellner and J. Sondow, Power-Sum Denominators, preprint, 2017, \url{https://arxiv.org/abs/1705.03857}.
   \bibitem{wal} A. Walfisz, \textit{Weylsche Exponentialsummen in der neueren Zahlentheorie}, Veb Deutscher Verlag der Wissenchaften, Berlin 1963.
\end{thebibliography}
\end{document}